\documentclass[preprint,12pt]{elsarticle} %%%draft,

\usepackage{amsthm,amssymb,amsmath,amsfonts}
\usepackage{eucal}  	%\mathcal

\def\Rset{\mathbb{R}}
\def\Cset{\mathbb{C}}
\def\Kset{\mathbb{K}}

\newtheorem{thm}{Theorem}[section]
\newtheorem{prop}[thm]{Proposition}
\newtheorem{cor}[thm]{Corollary}

\newtheorem*{thmA}{Theorem A}
\newtheorem*{thmB}{Theorem B}
\newtheorem*{thmC}{Theorem C}

\theoremstyle{definition}
\newdefinition{defn}[thm]{Definition}
\newdefinition{rem}[thm]{Remark}

\theoremstyle{remark}
\newtheorem{case}{Case}

\numberwithin{equation}{section}

\journal{}

\begin{document}

\begin{frontmatter}

\title{Relationships between different types of initial conditions for simultaneous root finding methods}
\author{Petko D. Proinov}
\ead{proinov@uni-plovdiv.bg}
\address{Faculty of Mathematics and Informatics, University of Plovdiv, Plovdiv 4000, Bulgaria}

\begin{abstract}
The construction of initial conditions of an iterative method is one of the most important problems in solving nonlinear equations. 
In this paper, we obtain relationships between different types of initial conditions that guarantee the convergence of iterative methods for simultaneous finding all zeros of a polynomial. In particular, we show that any local convergence theorem for a simultaneous method can be converted into a convergence theorem with computationally verifiable initial conditions which is of practical importance. 
Thus, we propose a new approach for obtaining semilocal convergence results for simultaneous methods via local convergence results. 
\end{abstract}


\end{frontmatter}

%%%%%%%%%%%%%%%%%%%%%%%%%%%%%%%%%%%%%%%%%%%%%%%%%%%%%
%
%			          Introduction
%
%%%%%%%%%%%%%%%%%%%%%%%%%%%%%%%%%%%%%%%%%%%%%%%%%%%%%

%Section 1
\section{Introduction and preliminaries}
\label{sec:Introduction-preliminaries}

Throughout this paper ${(\Kset,|\cdot|)}$ denotes an algebraically closed normed field, 
$\Kset[z]$ denotes the ring of polynomials over $\Kset$, and
the vector space $\Kset^n$ is equipped with the $p$-norm
${\|x\|_p = \left( \sum_{i = 1}^n |x_i|^p \right)^{1/p}}$
for some ${1 \le p \le \infty}$.

Let ${f \in \Kset[z]}$ be a polynomial of degree ${n \ge 2}$. 
We consider the zeros of $f$ as a vector in $\Kset^n$.
More precisely, a vector ${\xi \in \Kset^n}$ is said to be a \emph{root-vector} of $f$ if
\(
{f(z) = a_0 \prod _{i = 1} ^ n (z - \xi_i)}
\)
for all ${z \in \Kset}$, where ${a_0 \in \Kset}$.
Without doubt the most famous iterative method for simultaneously finding all the zeros of a polynomial $f$ is the Weierstrass method, 
which is defined by 
\begin{equation}  \label{eq:Weierstrass-iteration}
x^{k + 1}  = x^k - W_f(x^k ), \qquad k = 0,1,2,\ldots,
\end{equation}
where the Weierstrass correction ${W_f \colon \mathcal{D} \subset \Kset^n \to \Kset^n}$ is defined by
\begin{equation} \label{eq:Weierstrass-correction}
W_f(x) = (W_1(x),\ldots,W_n(x)) \quad\text{with}\quad
W_i(x) = \frac{f(x_i)}{a_0 \prod_{j \ne i} (x_i  - x_j)}
\quad (i = 1,\ldots,n),
\end{equation}
where $a_0$ is the leading coefficient of $f$ and $\mathcal{D}$ is the set of all vectors in $\Kset^n$ with distinct components.

Let us consider three classical convergence theorems for the Weierstrass method. 
In these theorems, we assume that $f$ is a complex polynomial of degree ${n \ge 2}$ which has only simple zeros, 
and that ${\xi \in \Cset^n}$ is a root-vector of $f$. Throughout the paper we use the function 
${\delta \colon \Kset^n \to \Rset_+}$ defined by
\(
	{\delta(x) = \min_{i \ne j} |x_i - x _j|}
\)
and the function ${d \colon \Kset^n \to \Rset^n}$ defined by
\begin{equation} \label{eq:d}
d(x) = (d_1(x),\ldots,d_n(x)), \quad\text{where}\quad d_i(x) = \min_{j \ne i} |x_i  - x_j| .
\end{equation}
		
%Theorem A
\begin{thmA}[Dochev \cite{Doc62b}] \label{thm:Dochev-1962}
If ${x^0 \in \Cset^n}$ is an initial guess such that
\begin{equation} \label{eq:Dochev-initial-conditions}
	\|x^0 - \xi\|_\infty < \frac{\sqrt[n-1]{2} - 1}{2 \sqrt[n-1]{2} - 1} \, \delta(\xi),
\end{equation}
then the Weierstrass iteration \eqref{eq:Weierstrass-iteration} %%is well defined and 
converges quadratically to $\xi$.
\end{thmA}

%Theorem B
\begin{thmB}[Wang and Zhao \cite{WZ91}] \label{thm:Wang-Zhao-1991}
If ${x^0 \in \Cset^n}$ is an initial guess such that
\begin{equation} \label{eq:Wang-Zhao-initial-conditions}
	\|x^0 - \xi\|_\infty < \frac{\sqrt[n-1]{2} - 1}{4 \sqrt[n-1]{2} - 3} \, \delta(x^0),
\end{equation}
then the Weierstrass iteration \eqref{eq:Weierstrass-iteration}  
converges to $\xi$.
\end{thmB}

%Theorem C
\begin{thmC}[Petkovi\'c, Carstensen and Trajkovi\'c \cite{PCT95}] \label{thm:Petkovic-Carstensen-Trajkovic-1995}
If ${x^0 \in \Cset^n}$ is an initial guess with distinct components such that
\begin{equation} \label{eq:PCT-initial-conditions}
	\|W_f(x^0)\|_\infty < \frac{\delta(x^0)}{5 n} \, ,
\end{equation}
then the Weierstrass iteration \eqref{eq:Weierstrass-iteration} %%%is well defined and 
converges to a root-vector of $f$.
\end{thmC}

Both sides of the initial condition of Theorem~A depend on the desired root-vector $\xi$ 
which is unknown.
The initial condition of Theorem~B also contains unknown data, but only in the left-hand side of \eqref{eq:Wang-Zhao-initial-conditions}.
Usually, we say that these results are rather of theoretical importance.
The initial condition of Theorem~C is of significant practical importance since it depends only on available data: the coefficients of $f$, the degree $n$ and the initial guess $x^0$.

Surprisingly, each of these theorems is a consequence of the previous one. 
It turns out that this situation is not accidental.
Among the other results, 
we prove that from any theorem of type A, we can obtain a theorem of type B as well as a theorem of type C. 
Besides, from any theorem of type B we can obtain a theorem of type C.

The main purpose of this paper is to show that any local convergence theorem for a simultaneous method 
can be converted into a convergence theorem with computationally verified conditions. 
In other words, in this area both local and semilocal convergence results are of significant practical importance.
Our results are based on a new localization theorem for polynomial zeros.

%%%%%%%%%%%%%%%%%%%%%%%%%%%%%%%%%%%%%%%%%%%%%%%%%%%%%%%%%%%%%%%%%%%%%
%
%     Initial conditions for convergence of simultaneous methods
%
%%%%%%%%%%%%%%%%%%%%%%%%%%%%%%%%%%%%%%%%%%%%%%%%%%%%%%%%%%%%%%%%%%%%

%Section 2
\section{Initial conditions for convergence of simultaneous methods}
\label{Initial-conditions-for convergence-of-simultaneous-methods}

For given vectors ${x \in \Kset^n}$ and ${y \in \Rset^n}$, we define in ${\Rset^n}$ the vector
\[
\frac{x}{y} = \left( \frac{|x_1|}{y_1},\ldots,\frac{|x_n|}{y_n} \right),
\]
provided that $y$ has no zero components. Given $p$ such that ${1 \le p \le \infty}$, 
we denote by $q$ the conjugate exponent of $p$, i.e. $q$ is defined by means of
\[
1 \le q \le \infty \quad\text{and}\quad 1/p + 1/q = 1.
\]

Let $f \in \Kset[z]$ be a polynomial of degree $n \ge 2$, ${\xi \in \Kset^n}$ be a root-vector of 
$f$, and $x \in \Kset^n$ be an initial guess of an iterative method for simultaneous finding all zeros of $f$.
Below $R = R(n,p)$ is a real number which depends only on $n$ and ${p}$.
The most used initial conditions of convergence theorems of simultaneous methods (see, e.g., \cite{Pet08,Pro15,SAK94}) 
can be categorized into three types.

%Definition 2.1
\begin{defn} 
An initial condition is said to be: 
\begin{enumerate}[(a)]
	\item 
of the first type if it can be represented in the form
 \begin{equation} \label{eq:initial-conditions-first-type}
	\left\| \frac{x - \xi}{d(\xi)} \right\|_p  \le R \quad\text{or}\quad \frac{\|x - \xi\|_p}{\delta(\xi)} \le R;
\end{equation}
	\item 
	of the second type if it can be represented in the form
 \begin{equation} \label{eq:initial-conditions-second-type}
	\left\| \frac{x - \xi}{d(x)} \right\|_p  \le R \quad\text{or}\quad \frac{\|x - x\|_p}{\delta(x)} \le R;
\end{equation}
	\item 
	of the third type if it can be represented in the form
 \begin{equation} \label{eq:initial-conditions-third-type}
	\left\| \frac{W_f(x)}{d(x)} \right\|_p  \le R \quad\text{or}\quad \frac{\|W_f(x)\|_p}{\delta(x)} \le R. 
\end{equation}	
\end{enumerate}
\end{defn} 

%Remark 2.2
\begin{rem} \label{rem:FIC}
In the following, we state all results in terms of ${\|(x - \xi) / d(\xi) \|_p}$, ${\| (x - \xi) / d(x) \|_p}$ and ${\| W_f(x) / d(x) \|_p}$,
but all results remain true if we replace these initial conditions by  
${\| (x - \xi) \|_p / \delta(\xi)}$, ${\| (x - \xi) \|_p / \delta(x)}$ and ${\| W_f(x) \|_p / \delta(x)}$, respectively. 
\end{rem}

%%%%%%%%%%%%%%%%%%%%%%%%%%%%%%%%%%%%%%%%%%%%%%%%%%%%%
%
%     Localization of polynomial zeros
%
%%%%%%%%%%%%%%%%%%%%%%%%%%%%%%%%%%%%%%%%%%%%%%%%%%%%%

%Section 3
\section{Localization of polynomial zeros}
\label{sec:Localization-of-polynomial-zeros}

In this section we obtain a new localization theorem for polynomial zeros, which plays an important role in our paper.

\medskip
The following proposition is an improvement of Proposition~8.4 of \cite{Pro15}.

%Proposition 3.1
\begin{prop} \label{prop:disks-disjoint}
Let ${E = \| u / d(x) \|_p}$, where ${x,u \in \Kset^n}$ and ${1 \le p \le \infty}$. 
Let ${c \ge 0}$ be such that ${b c E < 1}$, where ${b = 2^{1/q}}$.
Then the closed disks 
\begin{equation}  \label{eq:disks-disjoint}
D_i = \{z \in \Kset : |z - x_i| \le c \, |u_i| \}, \quad i = 1,2,\ldots,n,
\end{equation}
are mutually disjoint.
\end{prop}

\begin{proof}
From the definition of $d(x)$, H\"older's inequality and ${b c E < 1}$, we obtain for ${i \ne j}$,
\[
c (|u_i| + |u_j|) \le c \left( \frac{|u_i|}{d_i(x)} + \frac{|u_j|}{d_j(x)} \right) |x_i - x_j| \le 
b \, c \, E \, |x_i - x_j| < |x_i - x_j|,
\] 
which proves that the disks \eqref{eq:disks-disjoint} are mutually disjoint. 
\end{proof}

The following result is due to Braess and Hadeler \cite{BH73} in the case when $f$ is a complex polynomial, 
but the proof in the general case is the same.  

%Proposition 3.2
\begin{prop}[Braess and Hadeler \cite{BH73}] \label{thm:Braess-Hadeler-1973}
Let ${f \in \Kset[z]}$ be a polynomial of degree ${n \ge 2}$,
${x \in \Kset^n}$ be a vector with distinct components, and let ${\alpha_1,\ldots,\alpha_n}$ be positive numbers. 
Then the union of the disks
\begin{equation} \label{eq:Braess and Hadeler-disks}
	G_i = \left\{ z \in \Kset : |z - x_i| \le \frac{1}{\alpha_i} \sum^n_{j = 1} \alpha_j \, |W_j(x)| \right\}, \qquad i =1,\ldots,n,
\end{equation}
contains all the zeros of $f$. Besides, if the union of $m$ of these disks is disjoint from the union of the remaining disks, 
then it contains exactly $m$ zeros of $f$, counted with their multiplicity.  
\end{prop}

%Proposition 3.3
\begin{prop} \label{prop:localization}
Let ${f \in \Kset[z]}$ be a polynomial of degree ${n \ge 2}$ and 
${1 \le p \le \infty}$.
Suppose there exists ${x \in \Kset^n}$ with distinct components and ${c \ge 1}$ such that
\begin{equation} \label{eq:inequality-c}
	b \, c \, E_f(x) < 1 \quad\text{and}\quad \frac{1}{c} + \frac{a \, E_f(x)}{1 - c \, E_f(x)} \le 1,
\end{equation}
where ${b = 2^{1/q}}$, ${a = (n - 1)^{1/q}}$ and the function ${E_f \colon \mathcal{D} \to \Rset_+}$ is defined by
\( 
E_f(x) = \| W_f(x) / d(x) \|_p \, . 
\)  
Then $f$ has only simple zeros in $\Kset$ and the disks 
\begin{equation}  \label{eq:big-disks-disjoint}
D_i = \{z \in \Kset : |z - x_i| \le c \, |W_i(x)| \}, \quad i = 1,2,\ldots,n,
\end{equation}
are mutually disjoint and each of them contains exactly one zero of $f$.	
\end{prop}

\begin{proof}
By the first inequality in \eqref{eq:inequality-c} and Proposition~\ref{prop:disks-disjoint} with ${E = E_f(x)}$, we conclude that the disks \eqref{eq:big-disks-disjoint} are mutually disjoint. It remains to prove that each of these disks contains exactly one zero of $f$.
We assume that ${E_f(x) \ne 0}$, since the case ${E_f(x) = 0}$ is trivial.  
We divide the proof into two cases.
%Case1
\begin{case}   
Suppose that the second inequality in \eqref{eq:inequality-c} is strict. 
Let ${i \in \{1,\ldots,n\}}$ be fixed, and let $R_i$ be the radius of the disk $D_i$.
We have to prove that $D_i$ contains exactly one zero of $f$.  
From ${c \, E_f(x) < 1}$ and the definitions of $d(x)$ and $E_f(x)$, we obtain that for each ${j \ne i}$, 
\begin{equation} \label{eq:inequality-W}
	|x_i - x_j| - c \, |W_j(x)| \ge (1 - c \, E_f(x)) d_j(x) > 0.
\end{equation}
Consider the disks \eqref{eq:Braess and Hadeler-disks} with ${\alpha_1,\ldots,\alpha_n}$ defined as follows
 \begin{equation} \label{eq:inequality-numbers-alpha}
	\alpha_i = \frac{1}{c |W_i(x)|} \quad\text{ and }\quad \alpha_j = \frac{1}{|x_i - x_j|- c \, |W_i(x)|} \quad\text{for } \, j \ne i.
\end{equation}
Denote by ${r_1,\ldots,r_n}$ the radii of the disks ${G_1,\ldots,G_n}$, respectively.
It follows from \eqref{eq:inequality-numbers-alpha}, \eqref{eq:inequality-W}, H\"older's inequality and \eqref{eq:inequality-c} that
\[
\sum^n_{j = 1} \alpha_j \, |W_j(x)| = \frac{1}{c} + \sum_{j \ne i} \frac{|W_j(x)|}{|x_i - x_j|- c \, |W_i(x)|} 
\le  \frac{1}{c} + \frac{a \, E_f(x)}{ 1 - c \, E_f(x)} < 1.
\]
Therefore, ${r_j < 1 / \alpha_j}$ ${(j = 1,\ldots,n)}$ which is equivalent to the inequalities
\[
r_i < R_i \quad\text{and}\quad r_j + R_i < |x_i - x_j| \, \text{ for } j \ne i. 
\]
This means that the disk $G_i$ is a subset of the interior of $D_i$, and that $D_i$ is disjoint from each of the disks $G_j$ for ${j \ne i}$.
Then it follows from Proposition~\ref{thm:Braess-Hadeler-1973} that $D_i$ contains exactly one zero of $f$. 
Thus, if the second inequality in \eqref{eq:inequality-c} is strict, then the disks \eqref{eq:big-disks-disjoint} are mutually disjoint and each of them contains exactly one zero of $f$.  
\end{case}
%Case2
\begin{case}   
Suppose that the second inequality in \eqref{eq:inequality-c} is an equality. 
Without loss of generality we can assume that $f$ is monic.
Consider the monic polynomial 
\(
	{g(z) = t f(z) + (1 - t) \prod_{j = 1}^n (z - x_j)}
\)
of degree $n$, where ${t \in [0,1]}$ is a parameter.
Since ${W_g(x) = t \, W_f(x)}$ and ${E_g(x) = t \, E_f(x)}$, then for every ${t \in [0,1)}$,
\begin{equation} \label{eq:inequality-c-case2}
|W_g(x)| < |W_f(x)|, \quad b \, c \, E_g(x) < 1 \quad\text{and}\quad \frac{1}{c} + \frac{a \, E_g(x)}{1 - c \, E_g(x)} < 1.
\end{equation}
From this and Case~1, we conclude that each of the disks \eqref{eq:big-disks-disjoint} contains exactly one zero of $g$ 
provided that ${t \in [0,1)}$. This remains true also in the case ${t = 1}$, because the zeros of a polynomial in ${\Kset[z]}$ 
depend continuously on its coefficients (see, e.g., Lang \cite[pp.~43--41]{Lan94}). 
This completes the proof since ${g = f}$ for ${t =1}$.
\qedhere
\end{case}
\end{proof}

In what follows, for given ${a \ge 1}$ we define the real functions $\alpha$ and $\beta$ by 
\begin{equation} \label{eq:alpha}
	\alpha(t) = \frac{2}{1 - (a - 1) t + \sqrt{(1 - (a - 1) t)^2  - 4 t}} \quad\text{and}\quad \beta(t) = \frac{2}{1 - (a - 1) t} \, .
\end{equation}
Note that ${1 \le \alpha(t) \le \beta(t)}$ provided that ${0 \le t \le 1 / (1 + \sqrt{a})^2}$.  

%Theorem 3.4
\begin{thm} \label{thm:localization}
Let ${f \in \Kset[z]}$ be a polynomial of degree ${n \ge 2}$.
Suppose there exists a vector ${x \in \Kset^n}$ with distinct components such that
\begin{equation}  \label{eq:localization1}
E_f(x) = \left\| \frac{W_f(x)}{d(x)} \right\|_p  \le \frac{1}{(1 + \sqrt a)^2}
\end{equation}
for some $1 \le p \le \infty$, where ${a = (n - 1)^{1/q}}$. 
In the case when ${n = 2}$ and ${p = \infty}$ we assume that inequality 
\eqref{eq:localization1} is strict. 
Then $f$ has only simple zeros in $\Kset$. Besides, for any real number ${c \in [\alpha(E_f(x)),\beta(E_f(x))]}$ the disks  
\eqref{eq:big-disks-disjoint}
are mutually disjoint and each of them contains exactly one zero of $f$.	
\end{thm}

\begin{proof}
	Let ${b = 2^{1/q}}$.
	It is easy to show that ${b \le 1 + \sqrt{a}}$ with equality only if ${n = 2}$ and ${p = \infty}$. Then it follows from 
	\eqref{eq:localization1} that ${E_f(x) < 1 / (2 b + a -1)}$. From this and 
${c \le \beta(E_f(x))}$, we get
\(
	{b \, c \, E_f(x) \le b \, E_f(x) \, \beta(E_f(x)) < 1}, 
\)
which proves the first inequality in \eqref{eq:inequality-c}.
The assumption ${c \in [\alpha(E_f(x)),\beta(E_f(x))]}$ implies the second inequality in \eqref{eq:inequality-c}. 
Now the statement follows from Proposition~\ref{prop:localization}.
\end{proof}

%Remark 3.5
\begin{rem} \label{rem:example}
Note that the strictness assumption cannot be dropped from Theorem~\ref{thm:localization}. 
Indeed, if ${f(z) = z^2}$ and ${x = (-1,1) \in \Kset^2}$, then \eqref{eq:localization1} with ${p = \infty}$ is an equality, 
but Theorem~\ref{thm:localization} does not hold.
\end{rem}

%Corollary 3.6
\begin{cor} \label{cor:localization}
Let ${f \in \Kset[z]}$ be a polynomial of degree ${n \ge 2}$.
Suppose there exists a vector ${x \in \Kset^n}$ with distinct components such that
\begin{equation}  \label{eq:localization2}
E_f(x) = \left\| \frac{W_f(x)}{d(x)} \right\|_p  \le \frac{1}{2(a + 1)}
\end{equation}
for some $1 \le p \le \infty$, where ${a = (n - 1)^{1/q}}$. In the case ${n = 2}$ and ${p = \infty}$ we assume that the inequality 
\eqref{eq:localization2} is strict. 
Then $f$ has only simple zeros in $\Kset$. Besides, for any real number $c \in [\gamma(E_f(x)),\beta(E_f(x))]$,  
where ${\gamma(t) = 1 / (1 - (a + 1) t)}$,
the closed disks \eqref{eq:big-disks-disjoint} are mutually disjoint and each of them contains exactly one zero of $f$.
\end{cor}

\begin{proof}
It follows from Theorem~\ref{thm:localization} and the inequality ${\alpha(t) \le \gamma(t) \le \beta(t)}$ which holds for 
$0 \le t \le 1 / (a + 1)$.
\end{proof}

The next result generalizes and improves Corollary~1.1 of \cite{Pet08}.

%Corollary 3.7
\begin{cor}  \label{cor:localization-PHI}
Let ${f \in \Kset[z]}$ be a polynomial of degree $
{n \ge 2}$.
Suppose there exists a vector ${x \in \Kset^n}$ with distinct components such that
\begin{equation}  \label{eq:localization-PHI}
	\left\| \frac{W_f(x)}{d(x)} \right\|_p  \le R %%%C \le \frac{1}{2(a + 1)}
\end{equation}
for some $1 \le p \le \infty$ and ${0 \le R \le 1 / (2 a + 2)}$, where $a = (n - 1)^{1/q}$. In the case $n = 2$, $p = \infty$ and 
${R = 1 / (2 a + 2)}$, we assume that the second inequality in \eqref{eq:localization-PHI} is strict. 
Then $f$ has only simple zeros in $\Kset$ and the disks
\begin{equation}  %\label{eq:inclusion-disks-2}
D_i  = \left\{ z \in\Kset : |z - x_i| \le \frac{|W_i(x)|}{1 - (a + 1) R} \right\},
\quad i = 1,2,\ldots,n,
\end{equation}
are mutually disjoint and each of them contains exactly one zero of  $f$.
\end{cor}

\begin{proof}
It follows from Corollary~\ref{cor:localization} with ${c = \gamma(R)}$.
\end{proof}

%%%%%%%%%%%%%%%%%%%%%%%%%%%%%%%%%%%%%%%%%%%%%%%%%%%%%%%%%%%%%%%%%%%%%%%%%%%%%%%%%%%%%%%%%%%
%
%      Relationships between initial conditions of the first type and the second type
%
%%%%%%%%%%%%%%%%%%%%%%%%%%%%%%%%%%%%%%%%%%%%%%%%%%%%%%%%%%%%%%%%%%%%%%%%%%%%%%%%%%%%%%%%%%%

%Section 4
\section{Relationships between initial conditions of the first type and the second type}

In this section, we show how to convert each local convergence theorem of the first type into a local convergence theorem of the second type.

%Proposition 4.1
\begin{prop} \label{prop:inequality}
Let $u,v \in \Kset^n$ be two vectors with distinct components and let $1 \le p \le \infty$.
Then
\begin{equation} \label{eq:inequality}
\left\| \frac{u - v}{d(u)} \right\|_p  \ge 
\left( 1 - b \left\| \frac{u - v}{d(u)} \right\|_p \right) \left\| \frac{u - v}{d(v)} \right\|_p \, ,
\end{equation}
where ${b = 2^{1/q}}$.
\end{prop}

\begin{proof}
According to Proposition~5.2 of \cite{Pro15}, we have
\[
d_i(v) \ge \left( 1 - b \left\| \frac{u - v}{d(u)} \right\|_p \right) d_i(u), \qquad i = 1,\ldots,n,
\]
which can be written in the form
\[
\frac{1}{d_i(u)} \ge \left( 1 - b \left\| \frac{u - v}{d(u)} \right\|_p \right) \frac{1}{d_i(v)} \, . 
\]
Multiplying both sides of this inequality by ${|u_i - v_i|}$ and taking the $p$-norm, we get \eqref{eq:inequality}.
\end{proof}

%Theorem 4.2
\begin{thm} \label{thm:local-local}
Let ${\xi \in \Kset^n}$, where  ${n \ge 2}$. 
Suppose ${x \in \Kset^n}$ is a vector with distinct components such that
\begin{equation}  \label{eq:local1}
\left\| \frac{x - \xi}{d(x)} \right\|_p  \le \frac{R}{1 + b R}
\end{equation}
for some $1 \le p \le \infty$ and $R > 0$, where $b = 2^{1/q}$.
Then $\xi$ has pairwise distinct components and
\begin{equation}  \label{eq:local2}
\left\| \frac{x - \xi}{d(\xi)} \right\|_p  \le R.
\end{equation}
Besides, if the inequality \eqref{eq:local1} is strict, then \eqref{eq:local2} is strict too.
\end{thm}

\begin{proof}
From \eqref{eq:local1} and Proposition~5.2 of \cite{Pro15}, we conclude that $\xi$ has pairwise distinct components.
Applying Proposition~\ref{prop:inequality} with ${u = x}$ and ${v = \xi}$, and taking into account 
\eqref{eq:local1}, we obtain \eqref{eq:local2}.
\end{proof}

Note that using Theorem~\ref{thm:local-local} we can transform a convergence theorem of the first type into a convergence theorem of the second type. In other words, we can convert every local convergence theorem with initial conditions of the type 
\eqref{eq:local2} into a local convergence theorem with initial condition of the form \eqref{eq:local1}.

%Remark 4.3
\begin{rem} \label{rem:AB}
Let ${R = (\sqrt[n-1]{2} - 1) / (2 \sqrt[n-1]{2} - 1)}$. 
Applying Theorem~\ref{thm:local-local} to Theorem A we immediately get Theorem B.
\end{rem}

%%%%%%%%%%%%%%%%%%%%%%%%%%%%%%%%%%%%%%%%%%%%%%%%%%%%%%%%%%%%%%%%%%%%%%%%%%%%%%%%%%%%%%%%%%%
%
%      Relationships between initial conditions of the second type and the third type
%
%%%%%%%%%%%%%%%%%%%%%%%%%%%%%%%%%%%%%%%%%%%%%%%%%%%%%%%%%%%%%%%%%%%%%%%%%%%%%%%%%%%%%%%%%%%

%Section 5
\section{Relationships between initial conditions of the second type and the third type}

In this section we show how to obtain a semilocal convergence theorem for simultaneous methods from every local convergence theorems of the second type. More precisely, we give two theorems for converting any local theorem of the second type into a theorem with computationally verifiable initial conditions.

%Theorem 5.1
\begin{thm} \label{thm:localization-root-vector}
Let ${f \in \Kset[z]}$ be a polynomial of degree ${n \ge 2}$.
Suppose there exists a vector ${x \in \Kset^n}$ with distinct components such that
\begin{equation} \label{eq:localization-root-vector}
E_f(x) = \left\| \frac{W_f(x)}{d(x)} \right\|_p  \le \frac{1}{(1 + \sqrt a )^2} 
\end{equation}
for some ${1 \le p \le \infty}$, where ${a = (n - 1)^{1/q}}$. In the case ${n = 2}$ and ${p = \infty}$ we assume that the inequality in 
\eqref{eq:localization-root-vector} is strict. 
Then $f$ has only simple zeros and there exists a root-vector 
${\xi \in \Kset^n}$ of $f$ such that
\begin{equation} \label{eq:root-vector-existence}
\left\| \frac{x - \xi}{d(x)} \right\|_p  \le \frac{2 E_f(x)}{1 - (a - 1) E_f(x) + \sqrt{(1 - (a - 1) E_f(x))^2  - 4 E_f(x)}} \, .
\end{equation}
\end{thm}

\begin{proof}
It follows from Theorem~\ref{thm:localization} that $f$ has only simple zeros and the disks
\[
D_i  = \{ z \in\Kset : |z - x_i| \le \alpha(E_f(x)) \, |W_i(x)| \}, \qquad i = 1,2,\ldots,n, 
\]
are mutually disjoint and each of them contains exactly one zero of $f$.
This means that there is a root-vector ${\xi \in \Kset^n}$ of $f$ such that
\[
	|x_i - \xi_i| \le \alpha(E_f(x)) \, |W_i(x)|.
\]
Dividing both sides of this inequality by ${d_i(x)}$ and taking the $p$-norm, we get \eqref{eq:root-vector-existence}.
\end{proof}

%Theorem 5.2
\begin{thm} \label{thm:second-local-semilocal}
Let ${f \in \Kset[z]}$ be a polynomial of degree ${n \ge 2}$.  
Suppose there exists a vector ${x \in \Kset^n}$ with distinct components such that
\begin{equation}  \label{eq:second-local-semilocal-1}
\left\| \frac{W_f(x)}{d(x)} \right\|_p  \le \frac{R (1 - R)}{1 + (a - 1) R} 
\end{equation}
for some ${1 \le p \le \infty}$ and ${0 \le R \le 1/(1 + \sqrt a)}$, where ${a = (n - 1)^{1/q}}$.
In the case ${n = 2}$, ${p = \infty}$ and ${R = 1/(1 + \sqrt a)}$, we assume that inequality \eqref{eq:second-local-semilocal-1} is strict. 
Then $f$ has only simple zeros in $\Kset$ 
and there exists a root-vector ${\xi \in \Kset^n}$ of $f$ such that
\begin{equation}  \label{eq:second-local-semilocal-2}
\left\| \frac{x - \xi}{d(x)} \right\|_p \le R.
\end{equation} 
If inequality \eqref{eq:second-local-semilocal-1} is strict, then \eqref{eq:second-local-semilocal-2} is strict too.
\end{thm}

\begin{proof}
Let ${\tau = 1 / (1 + \sqrt{a})}$ and ${\mu = 1 / (1 + \sqrt{a})^2}$. 
Consider the real  function ${g \colon [0,\tau] \to [0,\mu]}$ defined by
\[
g(t) = \frac{t (1 - t)}{1 + (a - 1) t} \, . 
\]
Note that $g$ is strictly increasing on ${[0,\tau]}$. The inverse function of $g$ is the function ${h \colon [0,\mu] \to [0,\tau]}$ defined by
\[
h(t) = \frac{2 t}{1 - (a - 1) t + \sqrt{(1 - (a - 1) t)^2  - 4 t}}. 
\]
It follows from \eqref{eq:second-local-semilocal-1} and ${R \in [0,\tau]}$ that 
\[
E_f(x) = \left\| \frac{W_f(x)}{d(x)} \right\|_p  \le g(R) \le \mu = \frac{1}{(1 + \sqrt a )^2} \, .
\]
By Theorem~\ref{thm:localization-root-vector} we conclude that $f$ has only simple zeros and there exists 
a root-vector ${\xi \in \Kset^n}$ of $f$ such that
\[
\left\| \frac{x - \xi}{d(x)} \right\|_p  \le h(E_f(x)) \le h(g(R)) = R
\]
which proves \eqref{eq:second-local-semilocal-2}.
\end{proof}

%Corollary 5.3
\begin{cor} \label{cor:second-local-semilocal}
Let ${f \in \Kset[z]}$ be a polynomial of degree ${n \ge 2}$.  
Suppose there exists a vector ${x \in \Kset^n}$ with distinct components such that
\begin{equation}  \label{eq:second-local-semilocal-corollary-1}
\left\| \frac{W_f(x)}{d(x)} \right\|_p  \le \frac{R}{1 + (a + 1) R}
\end{equation}
for some ${1 \le p \le \infty}$ and ${0 < R \le 1/(1 + a)}$, where ${a = (n - 1)^{1/q}}$.
In the case ${n = 2}$, ${p = \infty}$, ${R = 1/(1 + a)}$, we assume that inequality \eqref{eq:second-local-semilocal-corollary-1} is strict. 
Then $f$ has only simple zeros in $\Kset$ and there exists a root-vector ${\xi \in \Kset^n}$ of $f$ which satisfies 
\eqref{eq:second-local-semilocal-2}. 
If inequality \eqref{eq:second-local-semilocal-corollary-1} is strict, then \eqref{eq:second-local-semilocal-2} is strict too.
\end{cor}

\begin{proof}
It is easy to show that if $x$, $p$ and $R$ satisfy the assumptions of Corollary~\ref{cor:second-local-semilocal}, then they satisfy the assumptions of Theorem~\ref{thm:second-local-semilocal}.
\end{proof}

Note that using Theorem~\ref{thm:localization-root-vector}, Theorem~\ref{thm:second-local-semilocal} or 
Corollary~\ref{cor:second-local-semilocal}, we can transform a convergence theorem of the second type into a convergence theorem of the third type. For example, using Theorem~\ref{thm:second-local-semilocal} we can convert every local convergence theorem with initial conditions of the form \eqref{eq:second-local-semilocal-2} into a semilocal convergence theorem with initial condition of the form  
\eqref{eq:second-local-semilocal-1} provided that ${0 < R \le 1/(1 + \sqrt a)}$.

%Remark 5.4
\begin{rem} \label{rem:BC}
For simplicity, we replace the right-hand side of \eqref{eq:Wang-Zhao-initial-conditions} by a smaller one 
${R \delta(x^0)}$, where ${R = 1 / (2 n + 2)}$. 
Applying Corollary~\ref{cor:second-local-semilocal} to Theorem~B, we get Theorem~C with 
an initial condition ${\|W(x^0)\|_\infty \le \delta(x^0) /(3 n + 3)}$.
\end{rem}

%%%%%%%%%%%%%%%%%%%%%%%%%%%%%%%%%%%%%%%%%%%%%%%%%%%%%%%%%%%%%%%%%%%%%%%%%%%%%%%%%%%%%%%%%%%
%
%      Relationships between initial conditions of the first type and the third type
%
%%%%%%%%%%%%%%%%%%%%%%%%%%%%%%%%%%%%%%%%%%%%%%%%%%%%%%%%%%%%%%%%%%%%%%%%%%%%%%%%%%%%%%%%%%%

%Section 6
\section{Relationships between initial conditions of the first type and the third type}

In this section, we obtain relationships between initial conditions of the first type and the third type.

%Theorem 6.1
\begin{thm} \label{thm:first-local-semilocal}
Let ${f \in \Kset[z]}$ be a polynomial of degree ${n \ge 2}$.  
Suppose there exists a vector ${x \in \Kset^n}$ with distinct components such that
\begin{equation}  \label{eq:first-local-semilocal-1}
\left\| \frac{W_f(x)}{d(x)} \right\|_p  \le \frac{R (1 + (b - 1) R)}{(1 + b R)(1 + (a + b - 1) R)}
\end{equation}
for some ${1 \le p \le \infty}$ and ${0 < R \le 1 / (1 - b + \sqrt a)}$, where ${a = (n - 1)^{1/q}}$ and ${b = 2^{1/q}}$. 
Then $f$ has only simple zeros in $\Kset$ and there exists a root-vector ${\xi \in \Kset^n}$ of $f$ such that
\begin{equation}  \label{eq:first-local-semilocal-2}
\left\| \frac{x - \xi}{d(\xi )} \right\|_p  \le R.
\end{equation}
If inequality \eqref{eq:first-local-semilocal-1} is strict, then \eqref{eq:first-local-semilocal-2} is strict too.
\end{thm}

\begin{proof}
Let ${\tilde{R} = R / (1 + b R)}$. Then ${\tilde{R} \le 1 / (1 + \sqrt{a})}$ and condition 
\eqref{eq:first-local-semilocal-1} takes the form
\[
\left\| \frac{W_f(x)}{d(x)} \right\|_p  \le \frac{\tilde{R} (1 - \tilde{R})}{1 + (a - 1) \tilde{R}} \, .
\]
This inequality is strict if ${n =2}$ and ${p = \infty}$. It follows from Theorem~\ref{thm:second-local-semilocal} that $f$ has only simple zeros and there exists a root-vector 
${\xi \in \Kset^n}$ of $f$ such that
\[
\left\| \frac{x - \xi}{d(x)} \right\|_p \le \tilde{R} = \frac{R}{1 + b R} \, .
\]
Applying Proposition~\ref{prop:inequality} with ${u = \xi}$ and ${v = x}$ and taking into account 
the last inequality, we obtain \eqref{eq:first-local-semilocal-2}.
\end{proof}

%Corollary 6.2
\begin{cor} \label{cor:first-local-semilocal}
Let ${f \in \Kset[z]}$ be a polynomial of degree ${n \ge 2}$.  
Suppose there exists a vector ${x \in \Kset^n}$ with distinct components such that
\begin{equation}  \label{eq:first-local-semilocal2-infty-1}
\left\| \frac{W_f(x)}{d(x)} \right\|_p  \le \frac{R}{1 + (a + b + 1) R}
\end{equation}
for some $1 \le p \le \infty$ and ${0 < R \le 1 / (a - b + 1) }$, where ${a = (n - 1)^{1/q}}$ and ${b = 2^{1/q}}$. 
Then $f$ has only simple zeros in $\Kset$ and there exists a root-vector ${\xi \in \Kset^n}$ of $f$ which satisfies 
\eqref{eq:first-local-semilocal-2}.
Besides, if inequality \eqref{eq:first-local-semilocal2-infty-1} is strict, then \eqref{eq:first-local-semilocal-2} is strict too.
\end{cor}

\begin{proof}
It is easy to show that if $x$, $p$ and $R$ satisfy the assumptions of Corollary~\ref{cor:first-local-semilocal}, then they satisfy the assumptions of Theorem~\ref{thm:first-local-semilocal}.
\end{proof}

Note that using Theorem~\ref{thm:first-local-semilocal} or Corollary~\ref{cor:first-local-semilocal}, we can transform a convergence theorem of the first type into a convergence theorem of the third type. For example, using Theorem~\ref{thm:first-local-semilocal}, 
we can convert every local convergence theorem with initial condition of the first type into a semilocal convergence theorem with initial condition of the third type, provided that ${0 < R \le 1 / (1 - b + \sqrt a)}$.  

%Remark 6.3
\begin{rem} \label{rem:AC}
Corollary~\ref{cor:first-local-semilocal} was stated without proof in \cite{PP14b}, where it was used for obtaining 
a semilocal convergence result for the two-step Weierstrass method. 
Another application of Corollary~\ref{cor:first-local-semilocal} can be found in \cite{PI14b}.
\end{rem}

%%%%%%%%%%%%%%%%%%%%%%%%%%%%%%%%%%%%%%%%%%%%%%
%
%             References
%
%%%%%%%%%%%%%%%%%%%%%%%%%%%%%%%%%%%%%%%%%%%%%%


\begin{thebibliography}{10}
\expandafter\ifx\csname url\endcsname\relax
  \def\url#1{\texttt{#1}}\fi
\expandafter\ifx\csname urlprefix\endcsname\relax\def\urlprefix{URL }\fi
\expandafter\ifx\csname href\endcsname\relax
  \def\href#1#2{#2} \def\path#1{#1}\fi

\bibitem{Doc62b}
K.~Dochev, 
Modified {N}ewton method for simultaneous approximation of all roots of a given algebraic equation, 
Phys. Math. J. Bulg. Acad. Sci. 5 (1962) 136--139 (Bulgarian).

\bibitem{WZ91}
D.~R. Wang, F.~G. Zhao, On the determination of the safe initial approximation for the {D}urand-{K}erner algorithm, 
J. Comput. Appl. Math. 38 (2014) 447--456.
\newblock \href {http://dx.doi.org/10.1016/0377-0427(91)90188-P}
{\path{doi:10.1016/0377-0427(91)90188-P}}.

\bibitem{PCT95}
M.~S. Petkovi{\'c}, C.~Carstensen, M.~Trajkovi{\'c}, Weierstrass formula and zero-finding methods, 
Numer. Math. 69 (1995) 353--372.
\newblock \href {http://dx.doi.org/10.1007/s002110050097}
{\path{doi:10.1007/s002110050097}}.

\bibitem{Pet08}
M.~Petkovi{\'c}, 
Point {E}stimation of {R}oot {F}inding {M}ethods, Vol. 1933 of Lecture Notes in Mathematics, Springer, Berlin, 2008.
\newblock \href {http://dx.doi.org/10.1007/978-3-540-77851-6}
{\path{doi:10.1007/978-3-540-77851-6}}.

\bibitem{Pro15}
P.~D. Proinov, General convergence theorems for iterative processes and applications to the {W}eierstrass root-finding method, 
arXiv: 1503.05243, 2015.

\bibitem{SAK94}
B.~Sendov, A.~Andreev, N.~Kjurkchiev, Numerical {S}olution of {P}olynomial {E}quations, in: Handbook of {N}umerical {A}nalysis, 
{V}ol.\ {III}, Elsevier, Amsterdam, 1994, pp. 625--778.
\newblock \href {http://dx.doi.org/10.1016/S1570-8659(05)80019-5}
{\path{doi:10.1016/S1570-8659(05)80019-5}}.

\bibitem{BH73}
D.~Braess, K.~P. Hadeler, 
Simultaneous inclusion of the zeros of a polynomial,
Numer. Math. 21 (1973) 161--165.

\bibitem{Lan94}
S.~Lang, 
Algebraic {N}umber {T}heory, 2nd Edition, 
Vol. 110 of Graduate Texts in Mathematics, Springer, New York, 1994.
\newblock \href {http://dx.doi.org/10.1007/978-1-4612-0853-2}
{\path{doi:10.1007/978-1-4612-0853-2}}.

\bibitem{PP14b}
P.~D. Proinov, M.~D. Petkova, 
Convergence of the two-point {W}eierstrass root-finding method, 
Japan J. Indust. Appl. Math. 31 (2014) 279--292.
\newblock \href {http://dx.doi.org/10.1007/s13160-014-0138-4}
{\path{doi:10.1007/s13160-014-0138-4}}.

\bibitem{PI14b}
P.~D. Proinov, S.~I. Ivanov, 
On the convergence of {H}alley's method for simultaneous computation of polynomial zeros, 
J. Numer. Math. (2015) in   press.

\end{thebibliography}
\end{document}